\documentclass[preprint,11pt,1p,twocolumn]{elsarticle}

\usepackage{amsmath,amssymb}
\usepackage{amsfonts}
\usepackage{mathrsfs}
\usepackage{color}

\newtheorem{proof}{Proof}

\newtheorem{proposition}{Proposition}

\newtheorem{lemma}{Lemma}
\newtheorem{remark}{Remark}

\def\begcen{\begin{center}}
\def\endcen{\end{center}}

\newcommand{\col}{ \mbox{col} }

\def\L1ac{${\cal L}_1$--AC}
\def\L1{{\cal L}_1}

\def\L2{{\cal L}_2}
\def\L2e{{\cal L}_{2e}}

\def\rea{\mathbb{R}}

\def\begequarr{\begin{eqnarray}}
\def\endequarr{\end{eqnarray}}
\def\begequarrs{\begin{eqnarray*}}
\def\endequarrs{\end{eqnarray*}}
\def\begarr{\begin{array}}
\def\endarr{\end{array}}
\def\begequ{\begin{equation}}
\def\endequ{\end{equation}}
\def\lab{\label}
\def\begdes{\begin{description}}
\def\enddes{\end{description}}
\def\begenu{\begin{enumerate}}
\def\begite{\begin{itemize}}
\def\endite{\end{itemize}}
\def\endenu{\end{enumerate}}

\def\lef[{\left[\begin{array}}
\def\rig]{\end{array}\right]}

\def\begcen{\begin{center}}
\def\endcen{\end{center}}
\def\begrem{\begin{remark}\rm}
\def\endrem{\end{remark}}

\def\l{\lambda}

\def\TAC{{\it IEEE Transactions of Automatic Control}}
\def\AUT{{\it Automatica}}

\def\IJC{{\it Int. J. of Control}}
\def\CSM{{\it IEEE Control Systems Magazine}}

\begin{document}

\begin{frontmatter}

\title{\Large \bf When is a Parameterized Controller Suitable for Adaptive Control?}
\author{Romeo Ortega and Elena Panteley\corref{cor1}}
\ead{\{ortega\}panteley@lss.supelec.fr}
\address{Laboratoire des Signaux et Syst\`emes, CNRS--SUPELEC, 91192
Gif--sur--Yvette, France.}

\begin{abstract}
In this paper we investigate when a parameterized controller, designed for a plant depending on unknown parameters, admits a realization which is {\em independent} of the parameters. It is argued
that adaptation is unnecessary for this class of parameterized controllers. We prove that standard model reference controllers (state and output--feedback) for linear time invariant systems with a
filter at the plant input admit a parameter independent realization. Although the addition of such a filter is of questionable interest, our result formally, and unquestionably, establishes the
deleterious effect of such a modification, which has been widely publicized in the control literature under the name ${\cal L}_1$--adaptive control.

\end{abstract}


\end{frontmatter}

%
\section{Introduction and Problem Formulation}
\lab{sec1}
%
The following question is addressed in this paper. Consider a linear time--invariant (LTI) parameterized plant
\begequarrs
\dot x & = & A(\theta) x + B(\theta)u\\
y & = & C(\theta) x,
\endequarrs
with state $x \in \rea^n$, $u\in \rea^m$, $y\in \rea^p$ the plant input and output, respectively, and $\theta \in \rea^q$ a vector of constant parameters; and an LTI parameterized controller
\begequarrs
\dot \xi & = & A_c(\theta) \xi + B_c(\theta)\lef[{c} y \\ r \rig]\\
u & = & C_c(\theta) \xi,
\endequarrs
with $\xi \in \rea^{n_\xi}$ and $r \in \rea^s$ some external reference. Their feedback interconnection generates a closed--loop transfer matrix
$$
Y(s)=P_{cl}(s,\theta)R(s).
$$
Under which conditions does there exists a partial change of coordinates
\begequ
\lab{chacoo}
\chi=T(\theta)\lef[{c} x \\  \xi \rig],
\endequ
with $T(\theta):\rea^{n_\xi \times (n+n_\xi)}$ such that the following holds true.
\begite
\item[(i)] The dynamics of the controller in the coordinates $\chi$ is {\em independent} of the parameters $\theta$.
\item[(ii)] The transfer matrix $P_{cl}(s,\theta)$ remains invariant.
\endite
A controller verifying these conditions is said to admit a parameter--independent realization or, in short, that it is a PIRC.

Our interest in this question stems from model reference adaptive control (MRAC) where it is assumed that the parameters of the plant $\theta$ are unknown but a parameterized controller that matches
a desired reference model is assumed to be known. The scheme is made adaptive replacing in the controller the unknown vector $\theta$ by an on--line estimate $\hat \theta$, which is generated via a
parameter identifier. The rationale behind this approach is, clearly, that if the estimated parameters $\hat \theta$ converge to their true value $\theta$ then---modulo some technical
conditions---the adaptive controller will achieve the control objective.\footnote{As is well--known \cite{SASBOD}, consistent parameter estimation is not necessary to achieve the control objective,
this is the so--called self--tuning property of adaptive control.}

From this perspective it is clear that for a PIRC there is no need to make the scheme adaptive! Indeed, we can simply plug--in the parameter--free controller, which will generate the ideal
closed--loop transfer matrix $P_{cl}(s,\theta)$, ensuring the control objective.

Interestingly, this very simple observation has been totally overlooked by the adaptive control community. In particular, a flood of publications---see \cite{HOVetal} and references therein---is
devoted to the analysis and design of ``adaptive" schemes, which turn out to be based on PIRC. The parameterized controller proposed in  \cite{HOVetal} consists of the addition of an LTI
filter to a classical state--feedback model reference controller (MRC). This simple modification to the state--feedback MRC is called  ``${\cal L}_1$--AC architecture" in \cite{HOVetal}.\\

The main contributions of the paper are the proofs of the following facts.
\begite
\item[R1] Standard state--feedback and output--feedback MRC are {\em not PIRC}.
\item[R2] Adding {\em any} LTI, strictly proper, input filter to the standard state--feedback MRC makes it a PIRC. Hence the ${\cal L}_1$--AC architecture is a PIRC.
\item[R3] Output--feedback MRC can be rendered a PIRC adding a {\em suitably chosen} LTI input filter.
\endite

The first result is rather obvious and, as seen below, the proof is straightforward. In view of the discussion above, the second result puts a serious question mark on the relevance of ${\cal L}_1$--AC. Since there
is no reason why we should like to add an input filter to output--feedback MRC the last fact is of little practical interest, but is given to underscore  the deleterious effect of adding input
filters to the plant.

In \cite{ortpan} R2 was established for the case of first--order plants with first order filters and a regulation objective, {\em i.e.}, $r$ constant. The generalization to $n$--th order plants was
reported in \cite{ortpancsm} for the case of stabilization, {\em i.e.}, $r=0$. To the best of our knowledge, these are the only two publications that have addressed the issues raised in this paper.

The remaining of the paper is organized as follows. State--feedback MRC, and it's filtered version, are studied in Section \ref{sec2}. Section \ref{sec3} is devoted to output--feedback MRC. The
analysis of state--feedback MRC is carried--out using state realizations of both, the plant and the controller. On the other hand, for the analysis of
output--feedback MRC it is more natural to use polynomial representations. We wrap--up the paper with some concluding remarks regarding the adaptive implementations of the various MRC in Section \ref{sec4}.
%
\section{State--feedback Model Reference Control}
\lab{sec2}
%
In its simplest version state--feedback MRC deals with single--input, LTI systems of the form
\begequ
\lab{pla}
\dot x = Ax + bu
\endequ
where $x\in \rea^n$ is assumed to be {\em measurable},
$$
A = \lef[{ccccc} 0 & 1 & 0 & \dots & 0  \\
                  0 & 0 & 1 & \dots & 0  \\
                   \vdots & \vdots & \vdots & \vdots & \vdots  \\
                    0 & 0 & 0 & \dots & 1  \\
                    -a_1 & -a_2 & -a_3 & \dots & -a_n
   \rig]
$$
where $a_i \in \rea,i\in \bar n:=\{1,\dots,n\}$ are unknown coefficients, and $b=e_n$---the $n$--th vector of the Euclidean basis.\footnote{This assumption is made to simplify the notation and
without loss of generality. See Remark R3 in \cite{ortpancsm}.} We are also given a reference model
$$
\dot x_m = A_m x_m + br
$$
where the state $x_m\in \rea^n$ and  $r\in\rea$ is a bounded reference, $A_m \in \rea^{n \times n}$ is the Hurwitz matrix
$$
A_m = \lef[{ccccc} 0 & 1 & 0 & \dots & 0  \\
                  0 & 0 & 1 & \dots & 0  \\
                   \vdots & \vdots & \vdots & \vdots & \vdots  \\
                      0 & 0 & 0 & \dots & 1  \\
                    -a^m_1 & -a^m_2 & -a^m_3 & \dots & -a^m_n
   \rig]
$$
with $a^m_i \in \rea_+,\;i\in \bar n$, designer chosen coefficients.

Defining the vector of {\em unknown} parameters
$$
\theta=\col(a_1 -a^m_1, a_2 -a^m_2,\dots, a_n-a^m_n),
$$
where $\col(\cdot)$ denotes column vector, it is clear that
\begequ
\lab{matequ}
A+b\theta^\top=A_m.
\endequ
Hence, invoking \eqref{matequ}, we can write \eqref{pla} in the equivalent parameterized system form
\begequ
\lab{sysx}
\dot x  =  (A_m- b \theta^\top) x +b  u.
\endequ
A parameterized controller that achieves the model matching objective is clearly
\begequ
\lab{pcsmrc}
u  =  \theta^\top x + r.
\endequ
This MRC is made adaptive adding an identifier that generates the estimated parameters, denoted $\hat \theta \in \rea^n$.

In \cite{HOVetal} it is proposed to {\em add a filter} at the plant input, that is, to compute $u$ via
$$
u=F(p)(\hat \theta^\top x + r)
$$
where $F(p)\in \rea(p)$, $p:={d \over dt}$, is strictly proper and stable. More precisely,
\begequ
\lab{fil}
F(p)={N_f(p)\over D_f(p)},
\endequ
where,
$$
D_f(p)  = \sum_{i=0}^{n_{D_f}}d_{fi} p^i,\; N_f(p)  = \sum_{i=0}^{n_{N_f}}n_{fi} p^i
$$
with $n_{D_f}>n_{N_f}$ and $D_f(p)$ and $N_f(p)$ are coprime with designer chosen coefficients. This simple modification to the state--feedback MRC is called  ``${\cal L}_1$--AC architecture" in
\cite{HOVetal}.

A state realization of the filtered state--feedback MRC is
\begequarr
\nonumber \dot \xi &= &A_f \xi + b_f (\theta^\top x + r)\\
u&=& c_f^\top \xi,
\lab{pcsl1}
\endequarr
where
$$
F(p)=c_f^\top(pI-A_f)^{-1}b_f,
$$
and $n_{\xi}=n_{D_f}$.\\

\begin{proposition}\em
\lab{pro1}
Consider the plant \eqref{sysx}. The classical state--feedback MRC \eqref{pcsmrc} is not a PIRC. On the other hand, for any filter \eqref{fil}, its filtered version \eqref{pcsl1} is a PIRC.
\end{proposition}

\begin{proof}\em
The fact that the static state--feedback \eqref{pcsmrc} is not a PIRC is obvious, as there is no way it can be implemented without knowledge of $\theta$.

We will prove now that applying to  \eqref{pcsl1} the partial coordinate transformation \eqref{chacoo}, with\footnote{Notice that the proposed $T$ is independent of $\theta$.}
$$
T=\lef[{ccc} b_f b^\top & | & I_{n_\xi}\rig],
$$
verifies the conditions (i) and (ii) of Section \ref{sec1}. First, we compute the transfer matrix of the plant \eqref{sysx} in closed--loop with  \eqref{pcsl1} as
$$
P_{cl}(s,\theta)=\lef[{cc} I_{n} &  0 \rig][sI_{(n+n_{\xi})}-A_{cl}(\theta)]^{-1}\lef[{c} 0 \\ b_c\rig],
$$
where
$$
A_{cl}(\theta)=\lef[{cc} A_m -b\theta^\top & bc_f^\top \\ b_f\theta^\top & A_f \rig],
$$
is the closed--loop system matrix. Now, the closed--loop system in the coordinates $\col(x,\chi)$ takes the form
\begequ
\lab{cloloo1}
\lef[{c} \dot x \\  \dot \chi \rig]={\cal T}A_{cl}(\theta){\cal T}^{-1}\lef[{c} x \\  \chi \rig]+ {\cal T}\lef[{c} 0 \\ b_c\rig]r
\endequ
where
$$
{\cal T}=\lef[{cc} I_n & 0 \\  b_f b^\top &  I_{n_\xi} \rig]
$$
is the full--coordinate transformation matrix. Now, the transfer matrix of  \eqref{sysx} in closed--loop with \eqref{cloloo1} verifies
\begequarrs
&&\lef[{cc} I_{n} &  0 \rig][sI_{(n+n_{\xi})}-{\cal T}A_{cl}(\theta){\cal T}^{-1}]^{-1}{\cal T}\lef[{c} 0 \\ b_c\rig]=\\
&&\lef[{cc} I_{n} &  0 \rig]{\cal T}[sI_{(n+n_{\xi})}-A_{cl}(\theta)]^{-1}\lef[{c} 0 \\ b_c\rig]=P_{cl}(s,\theta),
\endequarrs
proving that the closed--loop transfer matrix is invariant to the change of coordinates.

To complete the proof it remains to show that the controller dynamics in \eqref{cloloo1} is independent of $\theta$. Towards this end we compute
\begequarrs
{\cal T}A_{cl}(\theta){\cal T}^{-1}&=&\lef[{cc} I_n & 0 \\  b_f b^\top &  I_{n_\xi} \rig]\lef[{cc} A_m -b\theta^\top & bc_f^\top \\ b_f\theta^\top & A_f \rig] \lef[{cc} I_n & 0 \\  -b_f b^\top &
I_{n_\xi} \rig]\\
&=&\lef[{cc} I_n & 0 \\  b_f b^\top &  I_{n_\xi} \rig]\lef[{cc} A_m -b\theta^\top - bc_f^\top b_f b^\top & bc_f^\top \\ b_f\theta^\top -  A_fb_f b^\top & A_f\rig].
\endequarrs
The $\theta$--dependent term in the dynamics of $\chi$ is
$$
-b_fb^\top b\theta^\top+b_f\theta^\top,
$$
which is equal to zero because $b^\top b=1$.
\end{proof}
%
\section{Output--feedback Model Reference Control}
\lab{sec3}
%
The classical problem of output--feedback MRC deals with the single--input single--output LTI plant
\begequ
\lab{sys}
D(p)y=N(p)u
\endequ
where $y,\;u \in \rea$,
$$
D(p)  = \sum_{i=0}^{n}d_i p^i,\; N(p)  = \sum_{i=0}^{m}n_i p^i
$$
with the relative degree
$$
d := n - m \geq 1,
$$
and $D(p)$ and $N(p)$ are coprime---whose coefficients are unknown in MRAC.

We make the following assumptions
\begenu
\item[A.1] $n$ and $d$ are known.
\item[A.2] $n_m=1$.
\endenu
While A.1 is a classical assumption in MRC, we remark that A.2 pertains to the high--frequency gain of the plant, which is assumed here to be equal to one. This assumption is made, without loss of
generality, to simplify the notation.

The MRC objective is to asymptotically drive to zero the tracking error
\begequ
\lab{e}
e = y - { 1 \over D_m(p)} r
\endequ
where
$$
D_m(p)=\sum_{i=0}^{d}d_{mi} p^i
$$
is a Hurwitz polynomial and $r\in\rea$ is a bounded reference. Consistent with Assumption A.2 we take $d_{md}=1$.

Instrumental for the development of MRC is the lemma below, known as the direct control model reference parameterization, first established by Monopoli \cite{MON} and Astrom and Wittenmark
\cite{ASTWIT}, see also \cite{SASBOD} for a modern derivation of the result.

\begin{lemma}\em
Consider the plant (\ref{sys}) and the tracking error (\ref{e}). There exists a vector $\theta \in \rea^{2n}$ such that
$$
D_m(p)e =  u - \theta^\top  \phi + \epsilon_t
$$
where $\epsilon_t$ is an exponentially decaying term due to initial conditions,\footnote{These term will be omitted (without loss of generality) in the sequel.} and $\phi \in \rea^{2n}$ is the
regressor vector given by
\begequ
\lab{phi}
\phi={1 \over \lambda(p)}[u,\dot u, \dots, u^{(n-2)}, y,\dot y, \dots, y^{(n-2)}, \l(p) y, \l(p) r]^\top
\endequ
with
$$
\lambda(p)=\Sigma_{i=0}^{n-1}\lambda_i p^i,\;\lambda_{n-1}=1,
$$
a designer chosen Hurwitz polynomial.
\end{lemma}

Classical MRC consists of a state realization of the transfer matrix \eqref{phi} and the control signal
\begequ
\lab{umrc}
u = \theta^\top  \phi.
\endequ
We consider also adding a filter at the input of the plant, hence compute the control input via
\begequ
\lab{outfee}
u=F(p)(\theta^\top \phi).
\endequ
We will refer to this scheme as  filtered output--feedback MRC.
\begin{proposition}\em
\lab{pro2}
Consider the plant  (\ref{sys}) and the tracking error (\ref{e}).
\begite
\item[(i)] The output--feedback MRC \eqref{phi}, \eqref{umrc} is not a PIRC.
\item[(ii)] The filtered output--feedback MRC   \eqref{fil}, \eqref{phi}, \eqref{outfee} is a PIRC for any filter $F(p)$ verifying the conditions:
\begenu
\item[C1] $n_{N_f} \geq n_{D_f}+d$.
\item[C2] $D_f(s)-N_f(s)$ is a stable polynomial.\\
\endenu
\endite
\end{proposition}

\begin{proof}\em
The key question is whether we can eliminate the {\em parameter--dependent term} $\theta^\top \phi$ from the controller equations. Since the part of the controller dynamics that generates $\phi$,
that is \eqref{phi}, is independent of the parameters it plays no role in the analysis. In output--feedback MRC there is no remaining dynamics, but only the output map  \eqref{umrc}. The proof of
(i) follows then immediately.

To establish the proof of claim (ii) we analyze the error equations
\begequarr
\nonumber D_m(p)e & = &  u - \theta^\top  \phi \\
u & = & F(p)(\theta^\top \phi).
\lab{keyequ0}
\endequarr
The question then reduces to the analysis of the transfer function from $e$ to $u$ that results after eliminating $\theta^\top \phi$ from \eqref{keyequ0}. If this transfer function, denoted
$H_{eu}(s)$, is stable and proper the controller can be implemented without knowledge of $\theta$---nor signal differentiation---and the scheme is a PIRC. The transfer function of interest is
\begequ
\lab{heu}
H_{eu}(s)={N_fD_m \over N_f-D_f}(s),
\endequ
where we have used \eqref{fil}. Finally, claim (ii) follows noting that the conditions C1 and C2 ensure that the transfer function \eqref{heu} is proper and stable.
\end{proof}
%
\section{Concluding Remarks}
\lab{sec4}
%
It has been argued in the paper that it is unnecessary to add adaptation to a PIRC. The basic premise that justifies this statement is that the control action of the ideal, known--parameter
controller can be generated without knowledge of the parameters! However, against this conclusion, the following questions may be raised.
\begenu
\item Is the set of plants that can be stabilized with the adaptive version of the controller larger than the set stabilized by its LTI, parameter--independent realization?
\item Under which conditions  the adaptive controller actually converges to its LTI, parameter--independent realization?
\endenu
Answers to these questions have been provided for the ${\cal L}_1$--AC architecture in \cite{ortpan} and \cite{ortpancsm} for the case of $r=0$ and a first order filter. In particular, it has been
shown that the answer to the first question above is {\em negative}. This fact is established showing that if the characteristic polynomial of the closed--loop transfer matrix $P_{cl}(s,\theta)$ is
not Hurwitz then the  ${\cal L}_1$--AC, with a a state predictor--based estimator, generates unbounded trajectories. Furthermore, it has been shown that, if  the  ${\cal L}_1$--AC ensures
boundedness of trajectories, then it {\em always} converges  to its LTI, parameter--independent realization. This fact is proven writing the ${\cal L}_1$--AC in the form
\begequarrs
\dot \xi &= &A_f \xi + b_f (\theta^\top x + r)+ b_f \tilde \theta^\top x\\
u&=& c_f^\top \xi,
\endequarrs
where $\tilde \theta= \hat\theta -\theta$ is the parameter error. Then, it is shown that, if the trajectories are bounded, we always have
$$
\lim_{t \to \infty}|\tilde \theta^\top(t)x(t)|=0.
$$

As a final remark we bring to the readers attention the fact that it is not possible to establish similar results for filtered output--feedback MRC. Indeed, there are systems that can be stabilized
with MRAC that cannot be stabilized with an LTI controller. Also, there are scenarios where trajectories of the MRAC are bounded but the perturbation term $\tilde \theta^\top \phi$ does not converge
to zero.
%
\section*{Acknowledgment}
The first author thanks Alban Quadrat for a useful discussion on elimination theory of algebro--differential systems.
%


\begin{thebibliography}{aa}
%
\bibitem{ASTWIT}
K. Astrom and B. Wittenmark, On self--tuning regulators, \AUT, Vol 9, pp. 185--199, 1973.

\bibitem{HOVetal}
N. Hovakimyan, C. Cao, E. Kharisov, E. Xargay and I. Gregory, ${\cal L}_1$--adaptive control for safety-critical systems, \CSM, Vol. 31, no. 5, pp. 54-104, October 2011.

\bibitem{MON}
R. V. Monopoli, Model reference adaptive control with an augmented error, \TAC, AC-19 (4), Oct. 1974, pp. 474-484.

\bibitem{ortpan}
R. Ortega and E. Panteley, Comments on ${\cal L}_1$--adaptive control: Stabilization mechanism, existing conditions for stability and performance limitations, \IJC, (to appear).

\bibitem{ortpancsm}
R. Ortega and E. Panteley, Adaptation is unnecessary in ${\cal L}_1$--``adaptive" control, \CSM, (to appear).

\bibitem{SASBOD}
S. Sastry and M. Bodson, {\bf Adaptive Control; Stability, Convergence and Robustness}, Prentice-Hall, New Jersey, 1989.
%
\end{thebibliography}
\end{document}